\documentclass[12pt,reqno]{amsart}
\usepackage{amsmath, amsfonts, amsthm, amssymb}
\usepackage[all]{xy}
\usepackage{graphicx}
\usepackage[letterpaper,margin=1.5in]{geometry}

\theoremstyle{plain}
\newtheorem{theorem}{Theorem}[section]

\newtheorem{proposition}[theorem]{Proposition}

\theoremstyle{definition}
\newtheorem{definition}[theorem]{Definition}

\theoremstyle{remark}
\newtheorem{remark}[theorem]{Remark}

\newcommand{\CC}{\mathbb{C}}

\newcommand{\qx}{\overline{x}}

\newcommand{\XX}{\mathfrak{X}}

\newcommand{\uphi}{d^{\phi}}
\newcommand{\dphi}{d_\phi}
\newcommand{\dpi}{d_\pi}

\begin{document}

\title{Extensions of Poisson Structures on Singular Hypersurfaces}
\author{Aaron McMillan Fraenkel}

\begin{abstract}
Fix a codimension-1 affine Poisson variety $(X,\pi_X)\subseteq\CC^n$ with an isolated singularity at the origin.  We characterize possible extensions of $\pi_X$ to $\CC^n$ using the Koszul complex of the Jacobian ideal of $X$.  In the particular case of a singular surface, we show that there always exists an extension of $\pi_X$ to $\CC^3$.
\end{abstract}

\maketitle

\tableofcontents

\section{Introduction}
\label{sec:intro}

Given a singular affine Poisson variety $(X,\pi_X)\subseteq\CC^n$, one may ask when it is possible to find a Poisson structure on $\CC^n$ extending the Poisson structure $(X,\pi_X)$.  This question has an affirmative answer in a number of classical examples (e.g. ADE-singularities and the case of symplectic quotients with quadratic invariants, see \cite{Lerman93, AF}), while work done in \cite{Davis02,Egilsson95,AF} provide a number of examples for which no such extension exists.  However, little progress has been made in understanding why these examples have different extension properties; the size and (computational) intractability of these negative examples obstructs a thorough understanding of this non-extension phenomenon.   This note comes from an attempt to produce a simpler, lower-dimensional example of this non-extension phenomenon.

The results of this note build from the simple observation that, for a hypersurface $V(\phi)\subseteq\CC^n$, the associated Koszul complex of the jacobian ideal $J_\phi$ can be realized as $(\XX^\bullet,\dphi)$ on the vector space $\XX^\bullet$  of multi-derivations of $\CC[x_1,\ldots,x_n]$.  This complex can then be used in conjunction with the standard Gerstenhaber algebra structure on $\XX^\bullet$ to derive conditions on possible extensions of a Poisson bracket on $V(\phi)$.  That is:

{
\renewcommand{\thetheorem}{\ref{thm:main}}
\begin{theorem}
 Suppose $\beta\in \XX^2$ is an extension of $(X,\pi_X)$ to $\CC^n$ as a bi-derivation, where $X=V(\phi)$ has only isolated singularities.  Then there exists $X_2\in\XX^2$ and $X_3\in\XX^3$ such that $$\beta = \dphi X_3 +\phi X_2,$$  where $X_3$ satisfies the condition
$$[X_3,\dphi X_3] = \dphi Y_5 +\phi Y_4, \mbox{ where }  Y_4\in\XX^4 \mbox{ and } Y_5\in\XX^5.$$
Furthermore, $\dphi X_3 +\phi \tilde{X}_2$ is also an extension of $\pi_X$ for any choice of $\tilde{X}_2\in\XX^2$.
\end{theorem}
\addtocounter{theorem}{-1}
}

When one considers surfaces $V(\phi)\subseteq\CC^3$, the theorem above implies that Poisson brackets are {\it always} extendable:

{
\renewcommand{\thetheorem}{\ref{thm:dim3}}
\begin{theorem}
Suppose $(X=V(\phi),\pi_X)\subseteq\CC^3$ is a 2-dimensional Poisson variety with only isolated singularities.  Then, the Poisson bracket $\pi_X$ extends to a Poisson bracket $\beta$ on $\CC^3$.  Furthermore, $\beta$ has the form:
$$\beta = f\frac{\partial \phi}{\partial x_3}\frac{\partial}{\partial x_1}\wedge\frac{\partial}{\partial x_2} +
		f\frac{\partial \phi}{\partial x_2}\frac{\partial}{\partial x_3}\wedge\frac{\partial}{\partial x_1}+
		f\frac{\partial \phi}{\partial x_1}\frac{\partial}{\partial x_2}\wedge\frac{\partial}{\partial x_3}$$
where $f\in\CC[x_1,x_2,x_3]$.
\end{theorem}
\addtocounter{theorem}{-1}
}

Poisson structures of this form are very special and have been extensively studied (see, for example, \cite{Pichereau06,Laurent13}).


\section{Basics on Extensions of Poisson Structures}
\label{sec:basics}

We refer the reader to \cite{Laurent13,Vanhaecke01} for background on algebraic Poisson varieties.  Throughout, $(X,\pi_X)$ is an irreducible affine algebraic Poisson variety embedded in $\CC^n$.  Therefore, as a variety, $X$ is the vanishing locus $X=V(I)$ of a prime ideal $I\subseteq\CC[x_1,\ldots,x_n]$.  Additionally, the algebra of functions $A(X):=\CC[x_1,\ldots,x_n]/I$ on $X$ is equipped with a Poisson bracket $\pi_X$.  That is, $\pi_X\in {\rm Hom}(\wedge^2 A(X),\CC)$ is a skew-symmetric bi-derivation satisfying the Jacobi identity.

\begin{definition} An extension of the Poisson bracket $\pi_X$ on $X$ to $\CC^n$ is a Poisson bracket $\pi$ on $\CC[x_1,\ldots,x_n]$ such that the projection map $$\CC[x_1,\ldots,x_n]\twoheadrightarrow\CC[x_1,\ldots,x_n]/I$$ is a homomorphism of Poisson algebras.  
\end{definition}

In practice, constructing an extension of $(X,\pi_X)$ is often done in two steps: (1) extend $\pi_X$ as a bi-derivation to a bi-derivation $\pi$ on $\CC^n$ that doesn't necessarily satisfy the Jacobi identity, (2) determine which of these brackets satisfies the Jacobi identity.

Step one is always possible and results in the following description (see \cite{Laurent13} for details):
\begin{proposition}
\label{prop:ext}
Denote the images of $x_1,\ldots,x_n$ in $A(X)$ by $\qx_1,\ldots,\qx_n$. Let $\qx_{ij}:=\pi_X(\qx_i,\qx_j)\in A(X)$ and choose representatives $x_{ij}\in\CC[x_1,\ldots,x_n]$ of $\qx_{ij}$.  The choice of representatives defines an extension of $\pi_X$ to a skew-symmetric bi-derivation $\pi$ of $\CC[x_1,\ldots,x_n]$ given by 
$$\pi(x_i,x_j)  = x_{ij}.$$
The extension $\pi$ satisfies

\begin{align}
\pi(f, I ) \in I,   \label{eqn:ham} \\
\pi(f,\pi(g,h))+\pi(g,\pi(h,f))+\pi(h,\pi(f,g))\in I, \label{eqn:jac}
\end{align}
for all $f,g,h\in\CC[x_1,\ldots,x_n]$.  Moreover, any extension of $\pi_X$ comes from such a choice of representatives.  
\end{proposition}

To ensure that such an extension defines a Poisson bracket, one must try to refine the choice of representatives so that equation \ref{eqn:jac} is identically zero.  Due to bilinearity and the Leibniz rule, finding an extension is therefore equivalent to choosing representatives such that the associated bracket $\pi$ satisfies the Jacobi identity on the generators:
$$\pi(x_i,\pi(x_j,x_k))+\pi(x_j,\pi(x_k,x_i))+\pi(x_k,\pi(x_i,x_j))=0 \mbox{ for all } i,j,k=1...n$$

There are two general cases $X=V(I)\subseteq\CC^n$ in which solving this Jacobi condition (for a choice of representatives) is particularly interesting.

{\it Case 1}:  When $n=3,4$, the Jacobi identity condition is not over-determined.  If $n=3$, the Jacobi identity is a single equation in three unknowns.  If $n=4$, the Jacobi condition consists of four equations in four unknowns.  For $n>4$, the condition becomes (increasingly) overdetermined.

{\it Case 2}: When the number of generators of $I$ is small (roughly: ${\rm codim}(X)$ is small, at least when $X$ not far from being a complete intersection), the Jacobi identity can only fail in a very coherent way.  In this case, one is more likely to determine the solvability of the Jacobi condition.  

The case where $X$ is a hypersurface in $\CC^3$ and $\CC^4$ is considered in this note.


\section{The Complex of Multi-Derivations $\XX^\bullet$}

One can rephrase much of the theory of the Poisson geometry of affine varieties in terms of the complex of multi-derivations.  Specifically, the description in terms of multi-derivations will allow us to apply the well-developed theory of Koszul complexes to the problem of finding extensions to Poisson algebras.  As before, we refer the reader to \cite{Laurent13} for details.

As in section \ref{sec:basics}, $(X,\pi)$ is an affine Poisson variety with coordinate ring $A(X)=A$.  We denote the vector space of $k$-derivations of $A$ by $$\XX^k_A:={\rm Hom}(\wedge^k A,\CC).$$  These vector spaces form the Lichnerowicz-Poisson complex:

\[\xymatrix{
 \ldots \ar[r]^\dpi  & \XX_A^3 \ar[r]^\dpi & \XX_A^{2} \ar[r]^\dpi & \XX_A^{1} \ar[r]^\dpi & \XX_A^0  \ar[r] & 0. \\
}\]

The differential is given by $\dpi A = -[A,\pi]$, where  $A\in\XX^k_A$ and $[\cdot,\cdot]$ is the Schouten bracket.  That $\dpi$ is a differential follows easily from the fact that $\pi$ is a Poisson bracket if and only if $[\pi,\pi]=0$.

\begin{remark}
In the case of $X=\CC^n$, we will simply drop the $A$ and denote the complex of multi-derivations $\XX^\bullet$. 
\end{remark}

\subsection{Duality with the sheaf of differentials}

Denote the usual complex of K\"{a}hler forms of $A(X)=A$ by $(\Omega_A^\bullet,d)$.  As with multi-derivations, when $X=\CC^n$, we will simply write $(\Omega^\bullet,d)$.  There is a natural  $A$-linear pairing $\langle\cdot,\cdot\rangle:\; \XX^k_A \otimes_A \Omega^k_A \longrightarrow A$ given by 
$$\langle P, gdf_1\wedge\ldots\wedge df_k \rangle := g P(f_1,\ldots,f_k)$$

This pairing gives rise to two sets of $A$-linear maps:
$$\begin{array}{ll}
 \XX^k \to {\rm Hom}(\Omega^k_A,A) & \mbox{ is an isomorphism, while } \\
 \Omega^k \to {\rm Hom}(\XX^k_A,A)& \mbox{ is neither injective nor surjective in general.}
 \end{array}$$

In the case $X=\CC^n$, both maps are isomorphisms.  We make use of this pairing via the volume form $\nu:=dx_1\wedge\ldots\wedge dx_n$ to obtain an isomorphism $$ \nu^\flat : \XX^k \to \Omega^{n-k}$$
defined via contraction with the volume form: $$\langle \nu^\flat(A),B\rangle = \langle \nu, A\wedge B\rangle,$$ for all $A\in\XX^k$ and $B\in\XX^{n-k}$.  This isomorphism allows us to transport the usual differential on $(\Omega^\bullet,d)$ to $\XX^\bullet$.  That is, we have the following diagram:
\[\xymatrix{
(\Omega^\bullet,d): & 0 \ar[r] & \Omega^0\ar[r]^{d} & \Omega^1 \ar[r]^{d} & \ldots \ar[r]^{d} & \Omega^{n-1} \ar[r]^{d} & \Omega^n  \ar[r] & 0 \\
(\XX^\bullet,D_\nu): &  0 \ar[r] & \XX^n\ar[r]^{D_\nu} \ar[u]^{\nu^\flat} & \XX^{n-1} \ar[r]^{D_\nu} \ar[u]^{\nu^\flat} & \ldots \ar[r]^{D_\nu} \ar[u]^{\nu^\flat} & \XX^{1} \ar[r]^{D_\nu} \ar[u]^{\nu^\flat} & \XX^0  \ar[r]  \ar[u]^{\nu^\flat} & 0, \\
}\]
Where $D_\nu := (\nu^{\flat})^{-1}\circ d \circ \nu^\flat$ is the {\it divergence operator} with respect to $\nu$.

\subsection{The Koszul complex}

Now, we will specialize to the case where $X$ is an affine Poisson hypersurface $X=V(\phi)\subseteq\CC^n$ that has an isolated singularity at the origin.   The Koszul complex associated to the jacobian ideal $J_\phi=(\frac{\partial\phi}{\partial x_1},\ldots,\frac{\partial\phi}{\partial x_n})$ is exact when $k\neq n$ and takes the well-known form:
\[\xymatrix{
(\Omega^\bullet,\uphi): & 0 \ar[r] & \Omega^0\ar[r]^{\uphi} & \Omega^1 \ar[r]^{\uphi} & \ldots \ar[r]^{\uphi} & \Omega^{n-1} \ar[r]^{\uphi} & \Omega^n  \ar[r] & 0, \\
}\]
where the differential is defined by $\uphi\alpha = \alpha\wedge d\phi$, for $\alpha\in\Omega^k$.   

Next, we will use the isomorphism between $\Omega^\bullet$ and $\XX^\bullet$ induced by the volume form $\nu$ to realize the Koszul complex on $\XX^\bullet$:

\begin{proposition}  Let $(\Omega^\bullet,\uphi)$ be the Koszul complex as described above.  Then $$F:\;(\XX^\bullet,\dphi)\to(\Omega^\bullet,\uphi)$$ is an isomorphism, where $F(A)=(-1)^{n+1}\nu^\flat(A)$ for $A\in\XX^k$ and
\[\xymatrix{
(\XX^\bullet,\dphi): & 0 \ar[r] & \XX^n \ar[r]^\dphi & \XX^{n-1} \ar[r]^\dphi & \ldots \ar[r]^\dphi & \XX^{1} \ar[r]^\dphi & \XX^0  \ar[r] & 0 \\
}\]
has $A$-linear differential  $\dphi A=\langle  d\phi,A\rangle$ for $A\in\XX^k$.
\end{proposition}

\begin{proof}  This follows from a straight-forward computation in local coordinates.  Let $J=(i_1,\ldots,i_k)$ and define the two following multi-indices:
$$ \hat{J} = (1,\ldots,\hat{i_1},\ldots,\hat{i_k},\ldots,n), $$
that is $\hat{J}$ is the complement of $J$ in $(1,2,\ldots,n)$, preserving the order.
$$ J_k = (i_1,\ldots,\hat{k},\ldots,i_k),$$
that is $J_k$ is $J$ with the element $k$ removed.  Notice that $\hat{J_i}$ and $\hat{J}\cup\{i\}$ have the same indices; their order is related by a permutation of odd order if $n$ is even and of even order if $n$ is odd.

One must check that, $(\nu^\flat\circ\dphi )(A)=(-1)^{n+1}( \uphi\circ\nu^\flat)(A)$ for $A = A_J\frac{\partial}{\partial x^J}\in\XX^k$.

\begin{eqnarray*}
(\nu^\flat\circ\dphi )(A) &=& \nu^\flat\left(\dphi \left(A_J\frac{\partial}{\partial x^J}\right)\right) \\ 
			&=& \nu^\flat\left( \sum_{j\in J}A_J \frac{\partial\phi}{\partial x_j} \frac{\partial}{\partial x^{J_j}} \right) \\
			&=& \sum_{j\in J}A_J \frac{\partial\phi}{\partial x_j} d x^{\hat{J_j}} \\
			&=& (-1)^{n+1}\sum_{j\in J}A_J \frac{\partial\phi}{\partial x_j} dx^{\hat{J}} \wedge dx^{j} \\
			&=& (-1)^{n+1} \left(\sum_{j\in J}A_J dx^{\hat{J}} \right)\wedge d\phi \\
			&=& (-1)^{n+1}( \uphi\circ\nu^\flat)(A)
\end{eqnarray*}
\end{proof}

Thus, it follows that, when $V(\phi)$ has only isolated singularities, the complex $(\XX^\bullet,\dphi)$ is exact for $k\neq 0$.


\section{Extension of Poisson Hypersurfaces}

In this section, we will continue to consider the case where $X$ is an affine Poisson hypersurface $X=V(\phi)\subseteq\CC^n$ that has an isolated singularity at the origin.  As described in proposition \ref{prop:ext}, we will first concentrate on extending  $\pi_X$ to $\CC^n$ as a bi-derivation, then focus on the Jacobi condition.

Suppose that $\beta\in\XX^2$ is an extension of the Poisson bracket $\pi_X$ to a (not necessarily Poisson) bi-derivation on $\CC^n$.  As noted in section \ref{sec:basics}, such an extension always exists.  By proposition \ref{prop:ext}, it follows that  $\dphi\beta = \langle d\phi,\beta\rangle=\beta(\phi,\cdot)$ vanishes on $V(\phi)$.  That is, we have
$$\dphi\beta = \phi X_1, \mbox{ for some } X_1\in\XX^1.$$
Applying $\dphi$ to both sides of this equality gives:
$$0 = \dphi^2\beta = \dphi(\phi X_1) = \phi \dphi X_1. $$
Thus $\dphi X_1 = 0$ and $X_1$ is a cocycle.  Since $H_1(\XX^\bullet,\dphi)=0$, it follows that there exists an $X_2\in\XX^2$ such that $X_1=\dphi X_2$ and our original condition becomes $$\dphi\beta=\phi X_1 = \phi\dphi X_2 = \dphi(\phi X_2).$$

Therefore, we obtain the condition $\dphi(\beta-\phi X_2) =0$.  Exactness of $(\XX^\bullet,\dphi)$ at $k=2$ then gives the existence of an $X_3\in\XX^3$ such that $\dphi X_3 = \beta - \phi X_2$.  This implies:

\begin{proposition}
\label{prop:lem}
Suppose $\beta\in \XX^2$ is an extension of $(X=V(\phi),\pi_X)$ to $\CC^n$, where $V(\phi)$ has only isolated singularities.  Then there exists $X_2\in\XX^2$ and $X_3\in\XX^3$ such that  $$\beta = \dphi X_3 +\phi X_2.$$
\end{proposition}

\begin{remark} A few remarks that follow easily from the reasoning above:
\begin{enumerate}
\item $\beta = \dphi X_3 +\phi X_2$ is an extension of $\pi_X$ for {\it any} choice of $X_2$.  It follows that $\dphi X_3$ is an extension of $\pi_X$; the hamiltonian derivation of $\phi$ with respect to this extension is identically $0$ (in general, it is only $I$-valued).
\item Unlike $X_2$, one does not get to choose $X_3$; it is determined by the bracket $\pi_X$ on $X=V(\phi)$.
\end{enumerate}
\end{remark}

We can further refine our description of $X_3$ using the condition imposed on $X_3$ by the Jacobi identity.  By proposition \ref{prop:ext}, condition \ref{eqn:jac}, we know that the 3-derivation $[\beta,\beta]$ vanishes is $\langle \phi \rangle$-valued.   Using elementary properties of the Schouten bracket, $[\beta,\beta]$ becomes:
\begin{eqnarray*}  [\beta,\beta]  &=& [\dphi X_3,\dphi X_3] + 2[ \dphi X_3, \phi X_2 ] + [\phi X_2,\phi X_2] \\
						&=& -\dphi([X_3,\dphi X_3]) + 2\phi[\dphi X_3,X_2] + 2\phi(\dphi X_2)\wedge X_2 +\phi^2[X_2,X_2] \\
						&=& -\dphi([X_3,\dphi X_3]) + \phi\left(2[\dphi X_3,X_2] + 2(\dphi X_2)\wedge X_2 +\phi[X_2,X_2]\right).
						\end{eqnarray*}
As the second term is $\langle \phi \rangle$-valued for every choice of $X_2$, it follows that $\dphi([X_3,\dphi X_3])$ must also be $\langle \phi \rangle$-valued.  Thus, we have the following theorem:

\begin{theorem}
\label{thm:main}  
Suppose $\beta\in \XX^2$ is an extension of $(X,\pi_X)$ to $\CC^n$ as a bi-derivation, where $X=V(\phi)$ has only isolated singularities.  Then there exists $X_2\in\XX^2$ and $X_3\in\XX^3$ such that $$\beta = \dphi X_3 +\phi X_2,$$  where $X_3$ satisfies the condition
$$[X_3,\dphi X_3] = \dphi Y_5 +\phi Y_4, \mbox{ where }  Y_4\in\XX^4 \mbox{ and } Y_5\in\XX^5.$$
Furthermore, $\dphi X_3 +\phi \tilde{X}_2$ is also an extension of $\pi_X$ for any choice of $\tilde{X}_2\in\XX^2$.
\end{theorem}

\begin{remark} The condition $[X_3,\dphi X_3] = \dphi Y_5 +\phi Y_4$ is proved analogously to the proof of proposition \ref{prop:lem}.
\end{remark}

\begin{remark} Therefore, finding an extension of $\pi_X$ to $\CC^n$ as a Poisson bracket is equivalent to choosing $X_2\in\XX^2$ such that $$[ \dphi X_3 +\phi X_2, \dphi X_3 +\phi X_2]=0.$$
\end{remark}

\subsection{The case $n=3$}

A direct consequence of theorem \ref{thm:main} is that all Poisson structures on $X=V(\phi)\subseteq\CC^n$ in the $n=3$ case extend to Poisson structures on $\CC^3$.  The proof of the theorem relies on the low-dimensional nature of $\CC^3$.

\begin{theorem}
\label{thm:dim3}
Suppose $(X=V(\phi),\pi_X)\subseteq\CC^3$ is a 2-dimensional Poisson variety with only isolated singularities.  Then, the Poisson bracket $\pi_X$ extends to a Poisson bracket $\beta$ on $\CC^3$.  Furthermore, $\beta$ has the form:
$$\beta = f\frac{\partial \phi}{\partial x_3}\frac{\partial}{\partial x_1}\wedge\frac{\partial}{\partial x_2} +
		f\frac{\partial \phi}{\partial x_2}\frac{\partial}{\partial x_3}\wedge\frac{\partial}{\partial x_1}+
		f\frac{\partial \phi}{\partial x_1}\frac{\partial}{\partial x_2}\wedge\frac{\partial}{\partial x_3}$$
where $f\in\CC[x_1,x_2,x_3]$.
\end{theorem}

\begin{proof} By proposition \ref{prop:ext}, an extension of $\pi_X$ to a bi-derivation $\beta$ on $\CC^3$ always exists. Theorem \ref{thm:main} then implies that we may assume $\beta=\dphi X_3$ for some $X_3\in\XX^3$.  However, this choice of $\beta$ is actually a Poisson bracket:
$$[\beta,\beta]=[\dphi X_3,\dphi X_3]=\dphi( [X_3,\dphi X_3]) = 0,$$
where the last equality holds because $[X_3,\dphi X_3]\in\XX^4=\{0\}$, as the space of 4-derivations is zero in dimension 3.  Since $X_3$ is a top-dimensional multi-derivation, 
$$X_3=f\frac{\partial}{\partial x_1}\wedge\frac{\partial}{\partial x_2}\wedge\frac{\partial}{\partial x_3}$$ for some $f\in\CC[x_1,x_2,x_3]$.  A direct computation using the definition of $\dphi$ yields:
$$\beta=\dphi X_3 =  f\frac{\partial \phi}{\partial x_3}\frac{\partial}{\partial x_1}\wedge\frac{\partial}{\partial x_2} +
		f\frac{\partial \phi}{\partial x_2}\frac{\partial}{\partial x_3}\wedge\frac{\partial}{\partial x_1}+
		f\frac{\partial \phi}{\partial x_1}\frac{\partial}{\partial x_2}\wedge\frac{\partial}{\partial x_3}$$
as desired.
\end{proof}

\begin{remark}
These Poisson structures on $\CC^3$ have been studied extensively in \cite{Pichereau06}.  When $f=1$, they are exact and their Poisson (co)homology and deformations are explicitly computed.  When $f\neq 1$, these cohomology spaces are not as easily computed.  However, as can be easily checked, the space of Casimir functions of $\beta$ is generated by $\phi$.
\end{remark}

\subsection{The case $n=4$}

In the case of $\CC^4$, theorem \ref{thm:main} also benefits from low-dimensional simplifications.   However, unlike the $n=3$ case, it is not clear that any Poisson structure $(V(\phi),\pi_X)$ extends to a Poisson structure on $\CC^4$.

Any extension $\beta = \dphi X_3 + \phi X_2$ of $\pi_X$ satisfies:
$$ [ X_3,\dphi X_3 ] = \dphi Y_5 + \phi Y_4 = \phi Y_4 $$
as $\XX^5 = \{0\}$ in dimension 4.   Thus, $[X_3,\dphi X_3]$ is $\langle\phi\rangle$-valued.  Using the formula (due to Koszul)
$$[A,B] = (-1)^k D_\nu(A\wedge B) - (D_\nu A)\wedge B - (-1)^k A\wedge(D_\nu B)$$  where  $A\in\XX^j, \; B\in\XX^k,$
we see that
$$[X_3,\dphi X_3] = D_\nu(X_3\wedge \dphi X_3) - (D_\nu X_3)\wedge \dphi X_3 - X_3 \wedge ( D_\nu d_\phi X_3).$$
The first term $D_\nu(X_3\wedge \dphi X_3)$ is zero, as $X_3\wedge \dphi X_3\in\XX^5=\{0\}$.  The second two terms are equal:
\begin{eqnarray*} \dphi( X_3\wedge D_\nu X_3) &=& \dphi X_3 \wedge D_\nu X_3 + X_3\wedge \dphi D_\nu X_3   \\
 0 & = & D_\nu X_3 \wedge \dphi X_3 - X_3\wedge D_\nu\dphi X_3
\end{eqnarray*}

The first line above holds because $\dphi$ is a derivation of the wedge product.  The second line holds because \begin{enumerate}
\item $X_3\wedge D_\nu X_3 =0 \in\XX^5=\{0\}$, and
\item $D_\nu\dphi X_3 = -\dphi D_\nu X_3$.
\end{enumerate}

Therefore $[X_3,\dphi X_3] = -2 D_\nu X_3 \wedge \dphi X_3$ is $\langle\phi\rangle$-valued.

\begin{theorem} Suppose $\beta\in \XX^2$ is an extension (as a bi-derivation) of $(X=V(\phi),\pi_X)$ to $\CC^4$, where $V(\phi)$ has only isolated singularities.  Then there exists $X_2\in\XX^2$ and $X_3\in\XX^3$ such that $$\beta = \dphi X_3 +\phi X_2,$$  where $X_3$ satisfies the condition
$$D_\nu X_3 \wedge \dphi X_3 \mbox{ is } \langle \phi \rangle\mbox{-valued}.$$
\end{theorem}

\bibliography{references}

\bibliographystyle{plain}


\end{document}